\documentclass[11pt]{article}

\usepackage{amsmath, amsfonts, amssymb}
\usepackage{mathrsfs}
\usepackage{theorem}
\usepackage{booktabs}
\usepackage{bm}

\newtheorem{thm}{Theorem}[section]

{\theorembodyfont{\rmfamily} \newtheorem{rem}[thm]{Remark}}

\newcommand{\abs}[1]{\lvert#1\rvert}
\newcommand{\ra}{\rightarrow}
\newcommand{\dis}{\displaystyle}

\def\R{\mathbf R}

\def\N{\mathbf N}

\def\d{\text{\rm{d}}}
\def\E{\mathbf E}
\def\p{\mathbf P}

\newcommand{\tim}[2]{#1\leq\! s\!<\!#2}
\newcommand{\low}[1]{\!-\!(cs\!+\!#1)}
\newcommand{\upb}[1]{as\!+\!#1}

\newcommand{\fin}{\hspace*{\fill}\rule{0.3em}{1ex}}
\newenvironment{proof}{{\bf \noindent Proof.}}{\fin}

\numberwithin{equation}{section}

\begin{document}

\title{Nonlinear boundary crossing probabilities of Brownian motion with random jumps}

\author{Jinghai Shao ${}^a$ and Liqun Wang ${}^b$\thanks{Corresponding author}\\
{\small $a$: School of Mathematical Sciences, Beijing Normal University, Beijing, China}\\
{\small E-mail: shaojh@bnu.edu.cn}\\
{\small $b$:  Department of Statistics, University of Manitoba, Winnipeg, Manitoba, Canada R3T 2N2}\\
{\small E-mail: liqun\_wang@umanitoba.ca}
}


\maketitle


\begin{abstract}
We derive explicit formulas for probabilities of Brownian motion with jumps crossing linear or piecewise linear boundaries in any finite interval. We then use these formulas to approximate the boundary crossing probabilities for general nonlinear boundaries. The jump process can be any integer-valued process and jump sizes can have general distributions. Moreover, the jump sizes can be even correlated and/or non-identically distributed. The numerical algorithm is straightforward and easy to implement. Some numerical examples are presented.
\end{abstract}

\paragraph{Keywords:}
Boundary crossing probability; Brownian motion; Jump-diffusion process;
First hitting time; First passage time; Piecewise linear boundary.

\paragraph{AMS 2010 Subject Classification:}
Primary 60J65, 60J75, Secondary 60J60, 60J70

\section{Introduction}

Recently the first passage time (FPT) of jump-diffusion processes has drawn much attention in the literature, mainly due to its applications in finance and insurance. Some recent applications can be found in e.g. Cont and Tankov (2004), Kou and Wang (2004), Jiang and Pistorius (2008), Cai et al (2009), Chiarella and Ziogas (2009), Kudryavtsev and Levendorskii (2009), Bakshi and Panayotov (2010), Jeannin and Pistorisus (2010), Chi and Lin (2011), and Dong et al (2011). Given the important role of the FPT in many applications, however, the computation of the FPT densities for jump-diffusion processes turns out to be very challenging and so far very few exact solutions exist. In a remarkable work, Kou and Wang (2003) obtained the explicit form of the Laplace transform of the FPT for a Brownian motion (BM) with double-exponentially distributed jumps crossing a constant boundary. Dong et al (2011) obtained similar results for hyper-exponential jumps. Perry et al (2004) established the integral equation for the FPT density for BM with general jump sizes. In these proposed methods, the FPT density or probability have to be calculated through numerical Laplace inversion. More recently, Giraudo (2009) proposed an approximate solution for the FPT density for two-constant jump sizes. So far, no explicit form of the FPT distribution exists even for a constant boundary, let alone for general time-dependent boundaries.

The calculation of the FPT distribution for diffusion processes is a long-standing problem. It is well-known that explicit formulas exist only for some special processes and boundaries. For general problems, there are mainly two analytic approaches. In the first approach, one establishes a certain integral or differential equation for the FPT density, and then solves this equation numerically to obtain an approximate solution. For example, the tangent approximation and other image methods are used by Strassen (1967), Daniels
(1969, 1996), Ferebee (1982) and Lerche (1986), while a series expansion method is used by Durbin (1971, 1992), Ferebee (1983), Ricciardi et al. (1984), Giorno et al. (1989), and Sacerdote and Tomassetti (1996).

In the second approach, one first establishes an explicit formula for piecewise linear boundary crossing probabilities (BCP), and then uses this formula to approximate the BCP for general nonlinear boundaries. This method was first proposed by Wang and P\"otzelberger (1997) for one-boundary problem and later extended to two-boundary problems by Novikov, Frishling and Kordzakhia (1999) and P\"otzelberger and Wang (2001). It has also been used by Novikov, Frishling and Kordzakhia (2003), Borovkov and Novikov (2005), and Downes and Borovkov (2008). One advantage of this approach is that the numerical computation is straightforward and the approximation error can be assessed and  controlled. Moreover, this approach is easy to extend to deal with jump-diffusion processes, as is shown in this paper.

Specifically, let $(\Omega,(\mathcal F_t)_{t\geq 0}, \p)$ be a filtered probability space and $(B_t)_{t\geq 0}$ an $\mathcal F_t$-Brownian motion.  Further, let $(N_t)_{t\geq 0}$ be a Poisson point process with rate $\lambda$ and $(\eta_i)_{i\in \N}$ a family of random variables representing jump sizes. Suppose that $\big((B_t),\,(N_t),\, (\eta_i)\big)$ are mutually independent. Then we consider the boundary crossing probability for the process
\begin{equation}\label{1.1}
X_t = B_t + \sum_{i=1}^{N_t} \eta_i,
\end{equation}
where the second part of $X_t$ is a so-called compound Poisson process. In this paper we first use the BCP approach of Wang and P\"otzelberger (1997) and P\"otzelberger and Wang (2001) to derive explicit formulas for the BCP of $X_t$ crossing one- and two-sided linear boundaries respectively. Then we extend these formulas to piecewise linear boundaries and use them to approximate the BCP of $X_t$ for general nonlinear boundaries. One advantage of this method is that it can also deal with more general jump component in $X_t$. For instance,  $N_t$ can be any integer-valued stochastic process starting with $N_0=0$. Moreover, the jumps $(\eta_i)_{i\in \N}$ are not necessarily $i.i.d.$ as long as the joint distribution of $(\eta_1, \eta_2,\ldots,\eta_k)$ is known for each $k\in \N$. This provides much more flexible models for real applications where the subsequent jumps are allowed to be dependent or even follow different distributions. Finally, although in this paper we deal with processes of the form $(\ref{1.1})$ explicitly, our method can be easily extended to processes of the form $X_t = \mu(t) + \sigma B_t + \sum_{i=1}^{N_t} \eta_i$, because the drift $\mu(t)$ and diffusion parameter $\sigma$ can be absorbed into the boundaries.

This paper is organized as follows. In section 2 we derive explicit formulas for the BCP for one- and two-sided linear boundaries respectively. Then in section 3 we derive the formulas for piecewise linear boundaries and use them to approximate the BCP for general nonlinear boundaries. Finally, some numerical examples are given in section 4 and conclusions are in section 5.

\section{Explicit formulas for linear boundaries}

In this section we first consider the linear boundary crossing probability
\begin{equation}\label{2.1}
\p(\max_{0\leq s\leq t}\{ X_s-as-b\} < 0),
\end{equation}
where $a$ and $b> 0$ are constants. To simplify notation,  throughout the paper we denote $\E[ A]=\E[\mathbf{1}_A]=\p(A)$ for any measurable set $A$. Further, let
$$\mathcal F_N^t=\sigma\big(N_s,\,0\leq s<t\big),\ \mathcal F_B^t=\sigma\big(B_s,\,0\leq s<t\big).$$
A jump time is defined to be the time $t>0$ such that $\dis X_t\neq X_{t^-}$, where $X_{t^-}=\lim_{s\ra t^-} X_s$. Subsequently, we denote a jump time by $\tau$ and the corresponding jump height by $\eta=X(\tau)-X(\tau^-)$. Thus, for each $k\in \N$, the first $k$ jump times are
\begin{align*}
\tau_1&=\inf\{t> 0;\ N_t\neq N_{t^-}\},\ \\
\tau_2&=\inf\{t>\tau_1;\ N_t\neq N_{t^-}\},\\
&\!\ldots\ldots\ldots\ldots\\
\tau_k&=\inf\{t>\tau_{k-1};\ N_t\neq N_{t^-}\}.
\end{align*}
Here similarly $N_{t^-}:=\lim_{s\ra t^-} N_s$. For each $k\in \N$, let $F_k(t_1,\ldots,t_k)$ be the joint distribution function of the jump times $(\tau_1,\ldots, \tau_k)$ and $G_k(h_1,\ldots,h_k)$ be the joint distribution function of the corresponding jump heights $(\eta_1,\eta_2,\ldots,\eta_k)$. For each $t>0$, the support of $F_k$ is $\mathcal S_k^t=\big\{(s_1,\ldots,s_k);\ 0<s_1<\ldots<s_k<t\big\}$. Further, denote the distribution of $B_t$ as
\[
\Phi_t(x)= \frac{1}{\sqrt{2\pi t}}\int_{-\infty}^x e^{-\frac{y^2}{2t}}\,\d y.
\]
Then we have the following results.

\begin{thm}[One-sided linear BCP]
For any constants $a$ and $b>0$, it holds
\begin{equation}\label{2.2}
\begin{split}
&\ \p(\max_{0\leq s< t} \{X_s-as-b\} < 0)
= \sum_{k=0}^{\infty} \p(N_t=k)\int_{\mathcal S_k^t} \d F_k(t_1, t_1+t_2,\ldots, \sum_{i=1}^k t_i)\\
&\quad \cdot\int_{\R^k}\d G_k(h_1\cdots h_k)\int_{-\infty}^{\beta_1}\cdots\int_{-\infty}^{\beta_k}
\prod_{i=1}^k\big(1-e^{-2\frac{b_i(at_i+b_i-x_i)}{t_i}}\big)\\
&\quad\cdot \Big(\Phi\big(a\sqrt{t_{k+1}}+\frac {b_{k+1}}{\sqrt{t_{k+1}}}\big) -e^{-2ab_{k+1}}\Phi\big(a
\sqrt{t_{k+1}}-\frac{ b_{k+1}}{\sqrt{t_{k+1}}}\big)\Big)\,\d\Phi_{t_1}(x_1)\ldots\d\Phi_{t_k}(x_k),
\end{split}
\end{equation}
where $b_1= b, b_{i+1}=b_i+at_i-x_i-h_i, i=1,\,\ldots,k; \beta_i=\min\{at_i+b_i,a t_i+b_i-h_i\}, \ i=1,\ldots, k$; $t_{k+1}=t-\sum_{i=1}^k t_i>0$, and $\Phi(x):=\Phi_1(x)$. Further, if $N_t$ is a Poisson process with rate $\lambda$ and $(\eta_i)$ are $i.i.d.$ with distribution $G$, then
\begin{equation}\label{2.3}
\begin{split}
&\ \p(\max_{0\leq s< t}\big\{ X_s-as-b\big\} < 0)
= \sum_{k=0}^{\infty}  \lambda^ke^{-\lambda t}\int_0^t\d t_1\int_0^{t-t_1}\d t_2\cdots\int_0^{t-\sum_{i=1}^{k-1} t_i}\d t_k \\
&\quad \cdot\int_{\R^k}\d G(h_1)\cdots\d G(h_k) \int_{-\infty}^{\beta_1}\cdots\int_{-\infty}^{\beta_k}
\prod_{i=1}^k\Big(1-e^{-2\frac{b_i(at_i+b_i-x_i)}{t_i}}\Big)\\
&\quad\cdot \Big(\Phi\big(a\sqrt{t_{k+1}}+\frac{ b_{k+1}}{\sqrt{t_{k+1}}}\big) -e^{-2ab_{k+1}}\Phi\big(a
\sqrt{t_{k+1}}-\frac{ b_{k+1}}{\sqrt{t_{k+1}}}\big)\Big)\,\d\Phi_{t_1}(x_1)\ldots\d\Phi_{t_k}(x_k).
\end{split}
\end{equation}
\end{thm}

\begin{proof}
First, since $(B_t), (N_t)$ and $(\eta_i)$ are independent, we have
\begin{equation}\label{2.4}
\begin{split}
&\p\big(\max_{0\leq s< t}\{ X_s-as -b\}< 0\big)\\
&=\p\Big(\max_{0\leq s< t}\big\{B_s+\sum_{i=1}^{N_s}\eta_i-as-b\big\} < 0\Big)\\ 
&=\sum_{k=0}^\infty \p(N_t=k)\int_{\mathcal S_k^t} \d F_k\Big(t_1,\ldots,\sum_{i=1}^k t_i\Big)\int_{\R^k}\d G(h_1)\cdots\d G(h_k) \\ 
&\quad\cdot
\p\Big(\max_{0\leq s<t_1}\!\big\{ B_s\!-\!as\!-\!b\big\}< 0,\ldots,\max_{\sum_{i=1}^{k} t_i\leq s<t}\!\big\{ B_s\!-\!as\!-\!b\!+\!\sum_{i=1}^k \!h_i\big\}< 0\Big).
\end{split}
\end{equation}
In order to calculate the term in the integration, we recall the following well-known results on Brownian motion (Siegmund, 1986, P.\, 375, G. Deelstra,1994):
\begin{gather*}
\p\big(\max_{0\leq s<t} \big\{B_s-as-b \big\}\geq 0\big|B_t=x\big)=e^{\frac{-2b(at+b-x)}{t}}, \ \text{for}\ x<at+b,\\
\p(\max_{0\leq s<t}\big\{ B_s-as-b\big\}\geq 0)=1-\Phi(a\sqrt{t}+\frac b{\sqrt t})+e^{-2ab}\Phi(a\sqrt{t}-\frac b{\sqrt{t}}).
\end{gather*}
Then we have {\small
\begin{align*}
&\p\Big(\max_{0\leq s  <t_1}\!\big\{ B_s\!-\!as\!-\!b\big\}\!<\! 0,\ldots,\max_{\sum_{1}^{k} t_i\leq s<t}\!\big\{ B_s\!-\!as\!-\!b\!+\!\sum_{i=1}^k \!h_i\big\}\!<\! 0\Big)\\
&=\int_{-\infty}^{\beta_1}\p\big(\max_{0\leq s<t_1}\big\{ B_s-as-b\big\}< 0\big|B_{t_1}=x_1\big)\\
&\ \cdot\p\big(\max_{t_1\leq s<t_1+t_2}\!\!\big\{B_s\!-\!as\!-\!b\!+\!h_1\big\}< 0,\ldots,\!
\max_{\sum_{1}^k t_i\leq s<t}\!\big\{ B_s\!-\!as\!-\!b\!+\!\sum_{i=1}^k h_i\big\}\!<\! 0\Big|B_{t_1}=x_1\big)\,\d\Phi_{t_1}(x_1)\\
&=\int_{-\infty}^{\beta_1}\big(1- e^{-\frac{2b(at_1+b-x_1)}{t_1}}\big)\p\Big( \max_{0\leq s<t_2} \big\{B_s-as-b_2\big\}\!<\! 0,\ldots,\\
&\hspace{4cm} \max_{\sum_{i=2}^k t_i\leq s<t-t_1} \big\{B_s-as-b_2+\sum_{i=2}^k h_i\big\}\!<\! 0\Big)\,\d\Phi_{t_1}(x_1).
\end{align*}}
Iterating this calculation process yields
\begin{equation}\label{2.5}
\begin{split}
&\p\Big(\max_{0\leq s  <t_1}\!\big\{ B_s\!-\!as\!-\!b\big\}\!<\! 0,\ldots,\max_{\sum_{i=1}^{k}\!\! t_i\leq s<t}\!\big\{ B_s\!-\!as\!-\!b\!+\!\sum_{i=1}^k \!h_i\big\}\!<\! 0\Big)\\
& =\int_{-\infty}^{\beta_1}\cdots\int_{-\infty}^{\beta_k}
\prod_{i=1}^k\big(1-e^{-2\frac{b_i(at_i+b_i-x_i)}{t_i}}\big)
\cdot \Big(\Phi\big(a\sqrt{t_{k+1}}+\frac{ b_{k+1}}{\sqrt{t_{k+1}}}\big)\\ & \qquad -e^{-2ab_{k+1}}\Phi\big(a
\sqrt{t_{k+1}}-\frac {b_{k+1}}{\sqrt{t_{k+1}}}\big)\Big)\,\d\Phi_{t_1}(x_1)\cdots\d\Phi_{t_k}(x_k).
\end{split}
\end{equation}
Substituting (\ref{2.5}) into (\ref{2.4}) we obtain the desired result (2.2).
\end{proof}

\begin{rem}
Formulas (\ref{2.2}) and (\ref{2.3}) are infinite sums due to the number of jumps of the process $(X_s, s\geq 0)$ in interval $[0,t)$. In practical calculations they have to be truncated at finite terms. From (\ref{2.4}) it is easy to see that the truncation error at item $k=n$ is bounded by $\sum_{k=n}^\infty \p(N_t=k)$, which is easily controlled by the behaviour of the process $N_t$. For example, if $N_t$ is a Poisson process, then
$$\sum_{k=n}^\infty \p(N_t=k)=\sum_{k=n}^\infty\frac{(\lambda t)^k}{k!}e^{-\lambda t}\leq \frac{(\lambda t)^{n+1}}{(n+1)!}.$$
Therefore, in practice one can determine a number $n$ such that $\frac{(\lambda t)^{n+1}}{(n+1)!}$ is small enough, and then sum up the terms up to $n-1$. Moreover, unlike in other methods such as by Kou and Wang (2003), controlling the truncation error in this way does not depend on the distribution of jump sizes $\eta_i$. In other words, the distribution of $\eta_i$ does not impact the accuracy of (\ref{2.2}) and (\ref{2.3}) in real calculation.
\end{rem}

Our second main result is based on the well-known formulas of Anderson (1960) \cite[(4.24) (4.32)]{An}, that for any $t>0$ and $x<\delta_1 t+\gamma_1$,
\begin{align}\label{theta}\notag
&\vartheta(\gamma_1,\delta_1,\gamma_2,\delta_2; t,x) :=\p\big(\exists\, s<t, B_s\geq \delta_1s+\gamma_1, B_u> \delta_2 u+\gamma_2,\forall\, u\in[0,s]\big|B_t=x\big)\\ \notag
&=\sum_{r=1}^\infty\big\{e^{-(2/t)[r^2\gamma_1(\gamma_1+\delta_1t-x)+(r-1)^2\gamma_2(\gamma_2+\delta_2 t-x)-r(r-1)\{\gamma_1(\gamma_2+\delta_2 t-x)+\gamma_2(\gamma_1+\delta_1 t-x)\}]}\\
&\quad -e^{-(2/t)[r^2\{\gamma_1(\gamma_1+\delta_1t-x)+\gamma_2(\gamma_2+\delta_2 t-x)\}-r(r-1)\gamma_1(\gamma_2+\delta_2 t-x)-r(r-1)\gamma_2(\gamma_1+\delta_1 t-x)]}\big\}.
\end{align}
and
\begin{align}\label{chi} \notag
&\chi(\gamma_1,\delta_1,\gamma_2,\delta_2;t) :=\p\big(\exists\, s<t, B_s\geq \delta_1s+\gamma_1, B_u> \delta_2 u+\gamma_2,\forall\, u\in[0,s]\big)\\ \notag
&=1-\Phi\Big(\frac{\delta_1 t+\gamma_1}{\sqrt t}\Big)\\ \notag
&\quad +\sum_{r=1}^\infty\Big\{e^{-2[r\gamma_1-(r-1)\gamma_2][r\delta_1-(r-1)\delta_2]}\Phi\Big(\frac{\delta_1 t+2(r-1)\gamma_2-(2r-1)\gamma_1}{\sqrt t}\Big)\\ \notag
&\quad-e^{-2[r^2(\gamma_1\delta_1+\gamma_2\delta_2)-r(r-1)\gamma_1\delta_2-r(r-1)\delta_2]}\Phi\Big(\frac{\delta_1 t+2r\gamma_2-(2r-1)\gamma_1}{\sqrt t}\Big)\\ \notag
&\quad - e^{-2[(r-1)\gamma_1-r\gamma_2][(r-1)\delta_1-r\delta_2]}\Big[1-\Phi\Big(\frac{\delta_1 t-2r\gamma_2+(2r-1)\gamma_1}{\sqrt t}\Big)\Big]\\
&\quad +e^{-2[r^2(\gamma_1\delta_1+\gamma_2\delta_2)-r(r-1)\gamma_2\delta_1-r(r+1)\gamma_1\delta_2]}
\Big[1-\Phi\Big(\frac{\delta_1 t+(2r+1)\gamma_1-2r\gamma_2}{\sqrt t}\Big)\Big]\Big\}.
\end{align}

Then we have the following results.

\newpage

\begin{thm}[Two-sided linear BCP]
For any $a,\, c\in \R$ and $b,\, d>0$ such that $at+b>-(ct+d)$ for $t>0$, it holds
\begin{align}\label{prep}\notag
&\p(-(cs+d)<X_s<as +b, \tim{0}{t})\\ \notag
&=\sum_{k=0}^{\infty} \p(N_t=k)\int_{\mathcal S_k^t} \d F_k\Big(t_1,\cdots,\sum_{i=1}^k t_i\Big)\int_{\R^k}\d G( h_1,\ldots, h_k)\\ \notag
&=\int_{\alpha_1}^{\beta_1}\!\!\d\Phi_{t_1}(x_1)\cdots\int_{\alpha_k}^{\beta_k}\!\!\d\Phi_{t_k}(x_k)
\prod_{i=1}^{k}\Big(1-\vartheta(b_i,a,-d_i,c|t_i,x_i)-\vartheta(d_i,c,-b_i,-a|t_i,-x_i)\Big)\\
&\quad \cdot \Big(1-\chi\big(b_{k+1}, a,-d_{k+1},-c\big|t-\sum_{i=1}^k t_i\big)-\chi\big(d_{k+1},c,-b_{k+1},-a\big|t-\sum_{i=1}^k t_i\big)\Big),
\end{align}
where
\begin{gather*}
b_1=b,  b_{i+1}=at_i+b_i-x_i-h_i, i=1,\ldots,k;\\
\beta_i=\min\{at_i+b_i, at_i+b_i-h_i\}, i=1,\ldots, k;\\
d_1=d; d_{i+1}=ct_i+d_i+x_i+h_i,\ i=1,\ldots, k;\\
\alpha_i=\min\{ct_i+d_i, ct_i+d_i+h_i\},\ i=1,\ldots, k.
\end{gather*}
\end{thm}

\begin{proof}
Again by the independence of $(B_t)_{t\geq 0}$, $(N_t)_{t\geq 0}$ and $(\eta_i)_{i\in \N}$, we have
\begin{align}\label{pre}\notag
&\p(-(cs+d)<X_s<as +b, \tim{0}{t})\\ \notag
&=\p(\low d<B_s+\sum_{i=1}^{N_s}\eta_i<\upb b,\tim 0t)\\ \notag
&=\sum_{k=0}^{\infty} \p(N_t=k)\int_{\mathcal S_k^t} \d F_k\Big(t_1,\cdots,\sum_{i=1}^k t_i\Big)\int_{\R^k}\d G( h_1,\ldots, h_k)\p\Big(\low d<B_s<\upb b, \\
&\tim 0{t_1};\ldots;\low d<B_s+\sum_{i=1}^kh_i<\upb b, \tim{\sum_{i=1}^k t_i}{t}\Big).
\end{align}
To evaluate the term in the above integral, we deduce inductively
\begin{align}\label{bebi}\notag
&\p\Big(\low d<B_s<\upb b, \tim 0{t_1};\ldots;\low d<B_s+\sum_{i=1}^kh_i<\upb b, \tim{\sum_{i=1}^k t_i}{t}\Big)\\ \notag
&=\int_{\alpha_1}^{\beta_1}\p\big(\low d<B_s<\upb b, \tim 0{t_1}|B_{t_1}=x_1\big)\\ \notag
&\cdot \p\Big(\low d< B_s+h_1<\upb b, \tim{t_1}{t_1+t_2};\ldots;\low d<B_s\\ \notag
&+\sum_{i=1}^k h_i<\upb b, \tim{\sum_{i=1}^k t_i}{t}\big|B_{t_1}=x_1\Big)\d\Phi_{t_1}(x_1)\Big\}\\ \notag
&=\int_{\alpha_1}^{\beta_1}\d\Phi_{t_1}(x_1)\cdots\int_{\alpha_k}^{\beta_k}\d\Phi_{t_k}(x_k)
 \prod_{i=1}^{k}\p\big(\low {d_i} <B_s<\upb {b_i},\tim 0{t_i}\big|B_{t_i}=x_i\big)\\  \notag
 &\quad\cdot\p\big(\low {d_{k+1}}<B_s< \upb {b_{k+1}}, \tim 0{t-\sum_{i=1}^k t_i}\big)\\ \notag
 &=\int_{\alpha_1}^{\beta_1}\d\Phi_{t_1}(x_1)\cdots\int_{\alpha_k}^{\beta_k}\d\Phi_{t_k}(x_k)
 \prod_{i=1}^{k}\Big(1-\vartheta(b_i,a,-d_i,c\big|t_i,x_i)-\vartheta(d_i,c,-b_i,-a|t_i,-x_i)\Big)\\
 &\quad \cdot \Big(1-\chi\big(b_{k+1}, a,-d_{k+1},-c\big|t-\sum_{i=1}^k t_i\big)-\chi\big(d_{k+1},c,-b_{k+1},-a\big|t-\sum_{i=1}^k t_i\big)\Big),
\end{align}
where (\ref{theta}) and (\ref{chi}) were used in the last step. The result then follows by substituting (\ref{bebi}) into (\ref{pre}).
\end{proof}

\section{Piecewise linear and nonlinear boundaries}

In this section, we first consider process $X_t$ crossing a piecewise linear boundary. Suppose $b(s)$ is a linear function on each subinterval of the partition $0=s_0<s_1< \cdots < s_{n-1}<s_n=t$ on $[0, t]$. Since the probability that any jump time falls on some points $s_i$ is zero, without loss of generality we assume that any $k\geq 1$ realized jump times $\{u_i;\ i=1,2,...,k\}$ and points $\{s_i;\ i=1,...,n-1\}$ are distinct and therefore they form a partition $0<t_1<\cdots<t_m=t$, where $m = k+n$. Then analog to the proofs of Theorem 2.1 in section 2 and Theorem 1 of Wang and P\"otzelberger (1997) we can show the following result.

\begin{thm}[Piecewise linear BCP]
For the piecewise linear function $b(t)$ defined on the partition $0=s_0<s_1< \cdots < s_{n-1}<s_n=t$, it holds
\begin{equation}
\begin{split}
&\p(\max_{0\leq s< t}\{ X_s-b(s)\} < 0)
= \sum_{k=0}^{\infty} \p(N_t=k)\int_{\mathcal S_k^t} \d F_k(u_1,u_2,\ldots, u_k)\int_{\R^k}\d G_k(h_1,h_2,...,h_k)\\
&\quad \int_{-\infty}^{\beta_1}\cdots\int_{-\infty}^{\beta_m}\prod_{i=1}^m \big(1-e^{-2(b_{i-1}-x_{i-1})(b_i-x_i)/(t_i-t_{i-1})}\big)\\
&\quad \cdot \Big(\Phi\big(\frac{b(t)-x_m}{\sqrt{t-t_{m}}}\big) -\exp(-\frac{2(b(t)-b_m)(b_m-x_m)}{t-t_m})\Phi\big((\frac{b(t)-2b_m+x_m)}{\sqrt{t-t_m}}\big)\Big)\\
&\quad \d\Phi_{t_1}(x_1)\d\Phi_{t_2-t_1}(x_2-x_1)\ldots\d\Phi_{t_m-t_{m-1}}(x_m-x_{m-1}),
\end{split}
\end{equation}
where $b_i = b(t_i) - \sum_{j=1}^k\mathbf{1}_{\{t_i> u_j\}}h_j$ and $\beta_i=\min\{b_i, b_i-h_i\}$.
\end{thm}

Now we consider the BCP of process $X_t$ crossing a general boundary $b(t)$. We first write
\[
\p(\max_{0\leq s< t} \{X_s - b(s)\}< 0) = \E[\E(\max_{0\leq s< t} \{X_s - b(s)\}< 0 | \mathcal F_J^t)],
\]
where
\[
\mathcal F_J^t=\sigma\big(N_s,\,0\leq s<t, \tau_k, \eta(\tau_k), k\geq 1\big).
\]
For any $n\geq 1$ and a partition $0<s_1<s_2<\cdots<s_n=t$, let $b_n(s)$ be the piecewise linear function connecting the points $b(s_i), i=1,2,...,n$ and suppose
\[
\max_{0\leq s\leq t}\abs{b_n(s) - b(s)}\to\ 0, \mbox{ as } n\to\infty.
\]
Then by the continuity property of probability measure,
\[
\begin{split}
\lim_{n\to\infty}\p(\max_{0\leq s< t} \{X_s - b_n(s)\}< 0)
&= \E[\lim_{n\to\infty}\E(\max_{0\leq s< t} \{X_s -  b_n(s)\}< 0 | \mathcal F_J^t)]\\
&= \E[\E(\max_{0\leq s< t} \{X_s - b(s)\}< 0 | \mathcal F_J^t)] \\
&= \p(\max_{0\leq s< t} \{X_s - b(s)\}< 0).
\end{split}
\]
Furthermore, by the results of P\"otzelberger and Wang (2001) and Borovkov and Novikov (2005), with probability one, we have
\[
\p(\max_{0\leq s< t} \{X_s - b_n(s)\}< 0 | \mathcal F_J^t) - \p(\max_{0\leq s< t} \{X_s - b(s)\}< 0 | \mathcal F_J^t) = O(\frac{1}{n^2}),
\]
which implies that
\begin{equation}
\p(\max_{0\leq s< t} \{X_s - b_n(s)\}< 0) - \p(\max_{0\leq s< t} \{X_s - b(s)\}< 0) = O(\frac{1}{n^2}).
\end{equation}

Hence $(3.1)$ can be used to approximate the BCP for $X_t$ and general boundary $b(t)$.

\section{Numerical computation and examples}

In the previous sections, we provided various formulas for the BCP of jump-diffusion processes crossing linear or piecewise linear boundaries in the form of multiple integrals. Since these integrals are expectations of the integrants with respect to the probability distributions of jump process $N_t, \eta_i$ and the multivariate normal variates $B_{t_i}$, they can be evaluated straightforwardly using the Monte Carlo integration method. As has been shown in Wang and P\"otzelberger (1997) and P\"otzelberger and Wang (2001), this method is very stable and produces fairly accurate results. In the following we illustrate this method using some examples.

Specifically, we consider the Poisson jump process with rates $\lambda = 3$ and $\lambda = 0.01$ to approximate the no-jump scenario, and three different jump hight distributions:

(1) double exponential (DE) distribution with density
\begin{equation}
g(y)=p\cdot\eta_1e^{-\eta_1y}1_{\{y\geq0\}}+(1-p)\cdot\eta_2e^{\eta_2y}1_{\{y<0\}}
\end{equation}
and parameter values $p=0.5$, $\eta_1=1/0.10$ and $\eta_2=1/0.15$;

(2) exponential (Exp) distribution with mean $\lambda = 0.15$; and

(3) Bernoulli (Ber) distribution
\begin{equation}
g(y)=p\cdot1_{\{y=0.15\}}+(1-p)\cdot1_{\{y=-0.15\}}
\end{equation}
with $p=0.5$.

We calculate one-sided BCP for one constant $(b(t) =1)$, two linear $(b(t) =\pm 0.5t+1.5)$ and three nonlinear boundaries $(b(t) =1+t^2, \sqrt{1+t}, \exp(-t))$. For nonlinear boundaries, the approximating piecewise linear functions are based on $n=32$ equally spaced points in the interval $[0,t]$. The cases of two linear boundaries are also considered by Kou and Wang (2003).

We use the following procedure to compute the BCP: 1) generate the number of jumps $N_t=k$ from the Poisson($\lambda t$); 2) generate $k$ uniformly distributed jump times $u_1, u_2, ..., u_k$ in the interval $[0,t]$ and the jump heights $h_1, h_2, ..., h_k$ from the above distribution $G_k$; 3) generate the random sample from multivariate normal distribution of $B_{t_1}, B_{t_2}, ..., B_{t_m}$; 4) evaluate the integrant in (3.1) say; 5) repeat the above steps $N$ times and calculate the Monte Carlo sample mean and standard deviation of the estimated integrant.

In all cases, $t=1$ and the Monte Carlo replications are $N=200000$. For comparison, we also include the case with $\lambda = 0$. All numerical computations are carried out using computer package Matlab or R. R functions are available upon request.

The results for linear boundaries are given in Table 1, while those for nonlinear boundaries are in Table 2, where the Monte Carlo simulation standard errors are given in parentheses. The numerical results for constant and linear boundaries are consistent with those in the literature (Kou and Wang 2003).

\begin{table}[h!]
\caption{BCP for linear boundaries with simulation standard errors in parentheses.}
\begin{center}
\begin{tabular}{@{}llccc@{}}
\toprule
                &                           & $\lambda=0$        & $\lambda=0.01$& $\lambda=3$     \\
\midrule

$b(t)=1$        &        DE           & 0.682826           & 0.682841      & 0.685138                \\
                &                           & (0.000811)         & (0.000812)    & (0.000915)              \\
                &              Exp          & 0.682366           & 0.679808      & 0.106362                \\
                &                           & (0.000812)         & (0.000816)    & (0.000639)              \\
                &              Ber            & 0.682884           & 0.682599      & 0.667722                \\
                &                           & (0.000812)         & (0.000812)    & (0.000931)              \\
\midrule
$b(t)=0.5t+1.5$ &    DE           & 0.941229           & 0.942003      & 0.938798                \\
                &                           & (0.000404)         & (0.000401)    & (0.000466)              \\
                &              Exp          & 0.941859           & 0.938322      & 0.214821                \\
                &                           & (0.000401)         & (0.000420)    & (0.000885)              \\
                &              Ber            & 0.941156           & 0.941872      & 0.933017                \\
                &                           & (0.000405)         & (0.000402)    & (0.000487)              \\
\midrule
$b(t)=-0.5t+1.5$&    DE           & 0.740624           & 0.739806      & 0.742868                \\
                &                           & (0.000827)         & (0.000829)    & (0.000887)              \\
                &              Exp          & 0.738117           & 0.735223      & 0.120888                \\
                &                           & (0.000830)         & (0.000836)    & (0.000689)              \\
                &              Ber            & 0.739535           & 0.738522      & 0.727180                \\
                &                           & (0.000829)         & (0.000830)    & (0.000901)              \\
\bottomrule
\end{tabular}
\end{center}
\end{table}

\begin{table}[h!]
\caption{BCP for nonlinear boundaries with simulation standard errors given in parentheses.}
\begin{center}
\begin{tabular}{@{}llccc@{}}
\toprule
                &                           & $\lambda=0$        & $\lambda=0.01$ & $\lambda=3$    \\
\midrule

$b(t)=1+t^2$    &      DE           & 0.850851           & 0.852495      & 0.845398                \\
                &                           & (0.000766)         & (0.000762)    & (0.000779)              \\
                &              Exp          & 0.851666           & 0.847987      & 0.167030                \\
                &                           & (0.000763)         & (0.000772)    & (0.000823)              \\
                &              Ber        & 0.851581           & 0.852696      & 0.836923                \\
                &                           & (0.000764)         & (0.000762)    & (0.000796)              \\
\midrule
$b(t)=\sqrt{1+t}$&    DE           & 0.803963           & 0.803150      & 0.801449                \\
                &                           & (0.000859)         & (0.000860)    & (0.000865)              \\
                &              Exp          & 0.804617           & 0.797564      & 0.141257                \\
                &                           & (0.000857)         & (0.000870)    & (0.000768)              \\
                &              Ber       & 0.803450           & 0.803659      & 0.788627                \\
                &                           & (0.000860)         & (0.000859)    & (0.000885)              \\
\midrule
$b(t)=exp(-t)$  &      DE           & 0.439138           & 0.439502      & 0.454230                \\
                &                           & (0.001079)         & (0.001079)    & (0.001084)              \\
                &              Exp          & 0.439860           & 0.434360      & 0.056572                \\
                &                           & (0.001079)         & (0.001078)    & (0.000506)              \\
                &              Ber        & 0.437854           & 0.437174      & 0.430308                \\
                &                           & (0.001078)         & (0.001078)    & (0.001078)              \\

\bottomrule
\end{tabular}
\end{center}
\end{table}

\section{Conclusions}

We derived explicit formulas for probabilities of Brownian motion with jumps crossing linear or piecewise linear boundaries in any finite interval. These formulas can be used to approximate the boundary crossing probabilities for jump-diffusions and general nonlinear boundaries. Unlike the FPT density method based on Laplace transform, our method does not reply on particular distributions for the jump components. In fact, the jump process can be any integer-valued process and jump sizes can have general distributions. Moreover, the jump sizes can be even correlated and/or non-identically distributed. Therefore this method is much more flexible and has much wider applications. The numerical calculation is straightforward and easy to implement.

Finally we note that using the transformation method of Wang and P\"otzelberger (2007), it is possible to obtain the BCP for more general jump-diffusion processes of the form $X_t = Y_t + \sum_{i=1}^{N_t} \eta_i$, where $Y_t$ is a diffusion that can be expressed as a functional transformation of a BM.


\section*{Acknowledgment}

The authors thank Mr. Zhiyong Jin for his assistance in the initial calculation of numerical examples. Financial support from the Natural Sciences and Engineering Research Council of Canada (NSERC) is gratefully acknowledged. The first author's research is also supported by FANEDD (No. 200917), 985-project and Fundamental Research Funds for the Central Universities (China).


\end{document}